\documentclass[a4paper,12pt]{amsart}
\pdfoutput=1

\usepackage{graphicx}
\usepackage{amsmath}
\usepackage{amssymb}
\usepackage[active]{srcltx}
\usepackage{hyperref}

\newtheorem{thm}{Theorem}%
\newtheorem{lem}{Lemma}%
\theoremstyle{definition}

\theoremstyle{remark}
\theoremstyle{plain}

\def\QQ{{\mathbb Q}}
\def\PP{{\mathbb P}}
\def\RR{{\mathbb R}}

\def\ZZ{{\mathbb Z}}

\def\veca{{\text{\boldmath$a$}}}

\def\scrF{{\mathcal F}}

\def\scrS{{\mathcal S}}
\def\scrT{{\mathcal T}}

\def\fT{{\mathfrak T}}

\def\e{\mathrm{e}}

\def\ASL{\operatorname{ASL}}

\def\meas{\operatorname{meas}}

\title{Gaps between logs}
\author{Jens Marklof}
\author{Andreas Str\"ombergsson}
\address{School of Mathematics, University of Bristol,
Bristol BS8 1TW, U.K.\newline
\rule[0ex]{0ex}{0ex} \hspace{8pt}{\tt j.marklof@bristol.ac.uk}}
\address{Department of Mathematics, Box 480, Uppsala University,
SE-75106 Uppsala, Sweden\newline
\rule[0ex]{0ex}{0ex} \hspace{8pt}{\tt astrombe@math.uu.se}}
\date{6 January 2013}
\thanks{J.M.\ is supported by a Royal Society Wolfson Research Merit Award and ERC Advanced Grant HFAKT. A.S.\ is a Royal Swedish Academy of Sciences Research Fellow supported by
a grant from the Knut and Alice Wallenberg Foundation.}

\begin{document}

\begin{abstract}
We calculate the limiting gap distribution for the fractional parts of $\log n$, where $n$ runs through all positive integers. By rescaling the sequence, the proof quickly reduces to an argument used by Barra and Gaspard in the context of level spacing statistics for quantum graphs. The key ingredient is Weyl equidistribution of irrational translations on multi-dimensional tori. Our results extend to logarithms with arbitrary base; we deduce explicit formulas when the base is transcendental or the $r$th root of an integer. If the base is close to one, the gap distribution is close to the exponential distribution.
\end{abstract}

\maketitle

\newcounter{sectionno}\setcounter{sectionno}{0}

\stepcounter{sectionno}\S \arabic{sectionno}.
The gap distribution is a popular measure to quantify the degree of randomness in a given deterministic sequence. Rudnick and Zaharescu have shown \cite{Rudnick02} that the gap distribution of the fractional parts of $2^n \alpha$ ($n=1,\ldots,N$), for almost all $\alpha$, converges in the limit $N\to\infty$ to an exponential distribution---the gap distribution of a Poisson point process. ($2^n$ may be replaced by any lacunary sequence.) The same is expected to hold for the fractional parts of $n^k \alpha$, for any integer $k\geq 2$ and $\alpha$ of bounded type, but this remains unproven \cite{Rudnick98,Rudnick01,Marklof03,HeathBrown10}. Numerical experiments also suggest that the fractional parts of $n^\beta$ may have an exponential gap distribution provided $\beta\in(0,\frac12)\cup(\frac12,1)$. In the case $\beta=\frac12$, Elkies and McMullen \cite{Elkies04} have calculated an explicit formula for the limiting gap distribution, which is evidently not exponential. Remarkably, their distribution coincides with the limiting gap distribution of directions of points in the affine lattice $\ZZ^2+\veca$ for any fixed vector $\veca\notin\QQ^2$ \cite{Marklof10}. Both of these findings follow from the equidistribution of translates of certain curves on the homogeneous space $\ASL(2,\ZZ)\backslash\ASL(2,\RR)$.
In the present paper we show that also the fractional parts of $\log_b n$ have a non-exponential limiting gap distribution, which we calculate explicitly. Our derivation reduces quickly to an argument used by Barra and Gaspard \cite{Barra00} in the context of spectral statistics of quantum graphs. The key ingredient here is Weyl equidistribution on multi-dimensional tori, similar in spirit to \cite[Chap.~3]{Kac59}. 

%$$\star$$

\vspace{10pt}\stepcounter{sectionno}\S \arabic{sectionno}.
To state our main results, let us denote by $\xi_n$ the fractional parts of $\log_b n$. We denote by $\{\xi_{n,N}\}_{n=1,\ldots,N}$ the ordered set of the first $N$ elements of $\xi_n$, so that
\begin{equation}
0\leq \xi_{1,N} \leq \xi_{2,N} \leq \ldots \leq \xi_{N,N} < 1.
\end{equation}
For purely notational reasons it will be convenient to set $\xi_{N+1,N}:=\xi_{1,N}+1$, and also consider the gap between the first and last element mod 1; this additional gap has of course no effect on the existence and form of the limiting gap distribution. 

For given $N$, the gap distribution $P_N(s)$ is defined as the probability density on $\RR_{\geq 0}$,
\begin{equation}
P_N(s) = \frac1N \sum_{n=1}^N \delta(s-N(\xi_{n+1,N}-\xi_{n,N})) .
\end{equation}

We denote by $\scrT$ the set of transcendental numbers $b>1$; $b=\e$ is a natural example, cf.~Figure \ref{fig1}. The technique we present here also works for algebraic $b$, but in general leads to more intricate limit distributions.
In \S \ref{BINTEGER} we discuss the simple case of integer base $b$, and in \S \ref{NEWSEC} the case when the base is the $r$th root of an integer. 

\begin{figure}
\begin{center}
\includegraphics[width=0.7\textwidth]{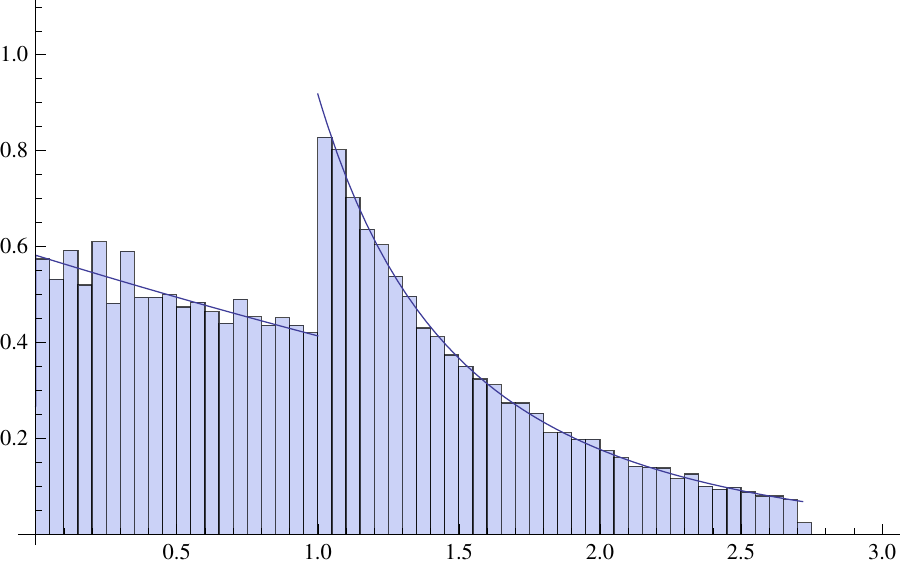}
\end{center}
\caption{The distribution of gaps between the fractional parts of $\log n$, where $n=1,\ldots,10^4$. The piecewise continuous curve is the limit distribution $P(s)$ of Theorem \ref{thm1} for $b=\e$.} 
%For $N=10^8$ data points, the histogram becomes virtually indistinguishable from the limit distribution \cite{numerics}.} 
\label{fig1}
\end{figure}

\begin{figure}
\begin{center}
\includegraphics[width=0.7\textwidth]{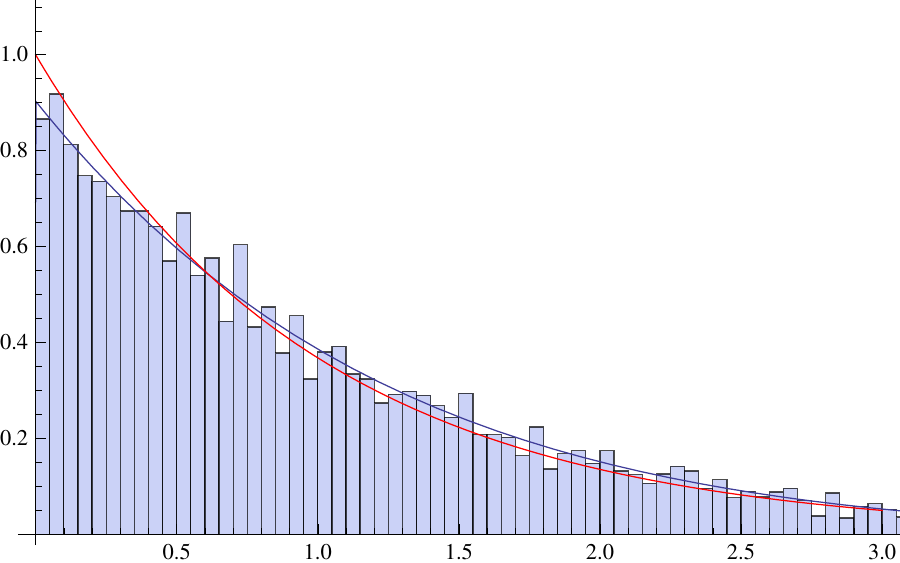}
\end{center}
\caption{The distribution of gaps between the fractional parts of $\log_b n$ with base $b=\e^{1/5}=1.221402758\ldots$, where $n=1,\ldots,10^4$. The blue curve is the limit distribution $P(s)$ of Theorem \ref{thm1}, the red curve is the exponential distribution $\e^{-s}$.} \label{fig2}
\end{figure}

\begin{thm}\label{thm1}
Let $b\in\scrT$. For any bounded continuous $f:\RR_{\geq 0} \to \RR$,
\begin{equation}
\lim_{N\to\infty}\int_0^\infty f(s)\, P_N(s)\, ds  = \int_0^\infty f(s)\, P(s)\, ds,
\end{equation}
where
\begin{equation}\label{Ps}
P(s)=\begin{cases}
\frac{1}{\log b}\, s^{-2} (F(s\log b)-F(sb^{-1}\log b)) & \text{if }\:0<s<\frac{1}{\log b} \\
-\frac{1}{\log b}\, s^{-2}F(sb^{-1}\log b) & \text{if }\: \frac{1}{\log b} < s<\frac{b}{\log b} \\
0 & \text{if }\: s>\frac{b}{\log b},
\end{cases}
\end{equation}
\begin{equation}
F(a) = a \frac{\partial}{\partial a} (a;b^{-1})- (a;b^{-1})
\end{equation}
and
\begin{equation}
(a;q) = \prod_{n=0}^\infty (1-aq^n)  \qquad \text{(q-Pochhammer symbol).}
\end{equation}
\end{thm}

The derivative of the Pochhammer symbol is explicitly
\begin{equation}
\frac{\partial}{\partial a} (a;q) = - (a;q) \sum_{j=0}^\infty \frac{1}{q^{-j}-a} .
\end{equation}
%so that
%\begin{equation}
%a \frac{\partial}{\partial a} (a;q)- (a;q) = - (a;q) \bigg( 1 + \sum_{j=0}^\infty \frac{a}{q^{-j}-a} \bigg).
%\end{equation}
Note that $P(s)$ has jump discontinuities at $s=\frac1{\log b}$ and $\frac b{\log b}$;
however both the left and right limits exist at these points, and are finite.
We also note that $\lim_{s\to0^+}P(s)=\frac{\log b}{b-1}$.

A more elementary formulation, which is equivalent to the above, is that for every $s\geq 0$
\begin{equation}
\lim_{N\to\infty} \frac{1}{N} \big|\{ n\leq N : N(\xi_{n+1,N}-\xi_{n,N})> s \}\big| = \int_s^\infty P(s')\, ds'. 
\end{equation}
We will show in the final paragraph \S \ref{LASTPARA} of this paper that, in the limit $b\to 1$, $P(s)$ converges to the exponential distribution $\e^{-s}$, cf.~Figure \ref{fig2}.

\vspace{10pt}\stepcounter{sectionno}\S \arabic{sectionno}.
To gain a better insight into the limiting gap distribution $P(s)$ for fixed $b$, let us recall that the fractional parts of $\log_b n$ are {\em not} uniformly distributed mod 1 \cite{Kuipers}. However, the sequence of $\eta_n$ $(n=1,\ldots,N$) given by the fractional parts of $\log_b (n/N)=\log_b n-\log_b N$ has a limit density mod 1, and evidently the same gap distribution (to see this, recall that we have added the gap between the first and last element of the sequence mod 1, and the ordering in the sequence remains the same). It is an easy exercise to show that for any interval $A\subset[0,1]$,
\begin{equation}
\lim_{N\to\infty} \frac1N \big|\{ n\leq N : \eta_n \in A \} \big| =  \int_A \rho(x) \, dx
\end{equation}
with density
\begin{equation}\label{rho}
\rho(x)= \frac{\log b}{b-1} \, b^x .
\end{equation}

%$$\star$$

\vspace{10pt}\stepcounter{sectionno}\S \arabic{sectionno}.
A more refined statement of Theorem \ref{thm1} is to consider the joint distribution of $\eta_n$ and the subsequent gap; define the corresponding probability density on $[0,1]\times\RR_{\geq 0}$ by
\begin{equation}
P_N(x,s) = \frac1N \sum_{n=1}^N \delta(x-\eta_{n,N}) \; \delta(s-N(\eta_{n+1,N}-\eta_{n,N})) ,
\end{equation}
where $\{\eta_{n,N}\}_{n=1,\ldots,N}$ is again the ordered set of the first $N$ elements of $\eta_n$. The joint distribution allows us to compare the statistics of gaps between elements near one point $x$ of the unit interval with the gap statistics at a second point $x'$.   

The explicit formula for the limit distribution will involve the density
\begin{equation}\label{Rab}
R(a,b) = 
(b^{-1};b^{-1})\,   \delta(a-1)+
\begin{cases}
\displaystyle
\frac{\partial^2}{\partial a^2} (a;b^{-1})  & \text{if $0<a<1$}\\[5pt]
 0 & \text{if $a>1$.}
\end{cases}
\end{equation}
The second derivative of the Pochhammer symbol is
\begin{equation}
\frac{\partial^2}{\partial a^2} (a;q) = (a;q) \sum_{\substack{j,k=0\\ j\neq k}}^{\infty} \frac{1}{(q^{-j}-a)(q^{-k}-a)} .
\end{equation}

\begin{thm}\label{thm2}
Let $b\in\scrT$. For any bounded continuous $f:[0,1]\times\RR_{\geq 0} \to \RR$,
\begin{equation}\label{thm2res}
\lim_{N\to\infty} \int_0^1 \int_0^\infty f(x,s)\, P_N(x,s)\,ds\,dx  =\int_0^1 \int_0^\infty f(x,s)\, P(x,s)\,ds\,dx ,
\end{equation}
where
\begin{equation}
P(x,s)=(b^{x-1}\log b)^2 R(s b^{x-1}\log b, b) .
\end{equation}
\end{thm}

%$$\star$$

\vspace{10pt}\stepcounter{sectionno}\S \arabic{sectionno}.
To unravel the previous statements further, let us consider a strictly increasing function $[0,1]\to[0,1]$ which maps the sequence $\eta_j$ to a new sequence $\tilde\eta_j$ which now is uniformly distributed mod 1. This function is given by (cf.~\cite{nato})
\begin{equation}\label{TMAP}
x \mapsto  \int_0^x \rho(y) \, dy = \frac{b^x-1}{b-1},
\end{equation}
and thus
\begin{equation}
\tilde\eta_n = \frac{b^{\eta_n}-1}{b-1} , \qquad 
\tilde\eta_{n,N} = \frac{b^{\eta_{n,N}}-1}{b-1} .
\end{equation}
We have, explicitly, for $n\in[N b^{-k}, N b^{-k+1})$ and $k\in\ZZ_{\geq0}$,
\begin{equation}\label{keyeq}
\tilde\eta_n  = \frac{b^k n - N}{N(b-1)}  .
\end{equation}

We will see in \S \ref{THM2EQTHM3para} that the knowledge of the joint distribution $P_N(x,s)$ for the original sequence and the analogue for the rescaled and ordered sequence,
\begin{equation}
\tilde P_N(x,s) = \frac1N \sum_{n=1}^N \delta(x-\tilde \eta_{n,N}) \; \delta(s-N(\tilde\eta_{{n+1},N}-\tilde\eta_{n,N})) 
\end{equation}
are equivalent: 

\begin{thm}\label{thm3}
Let $b\in\scrT$. For any bounded continuous $f:[0,1]\times\RR_{\geq 0} \to \RR$,
\begin{equation}\label{thm3res}
\lim_{N\to\infty} \int_0^1 \int_0^\infty f(x,s)\, \tilde P_N(x,s)\,ds\,dx  =\int_0^1 \int_0^\infty f(x,s)\, \tilde P(s)\,ds\,dx ,
\end{equation}
where
\begin{equation}\label{thm3tildePdef}
\tilde P(s) = (1-b^{-1})^2 R\big((1-b^{-1}) s, b\big) .
\end{equation}
\end{thm}

Note that the limit distribution is independent of $x$. This is not a consequence of  $\tilde\eta_j$ being uniformly distributed mod 1, but reflects the fact that the gap distribution is the same if we restrict the sequence to an arbitrary subinterval of $[0,1]$. 

We will first prove Theorem \ref{thm3} in \S \ref{THM3pfsec1} and \S \ref{THM3pfsec2}, and then retrace our steps by subsequently showing in \S \ref{THM2EQTHM3para} that Theorem \ref{thm3} implies Theorem \ref{thm2}, where
\begin{equation}\label{Pxs}
P(x,s) = \tilde P(\rho(x) s)  \rho(x)^2 =(b^{x-1}\log b)^2 R(s b^{x-1}\log b, b).
\end{equation}
We obtain Theorem \ref{thm1} from Theorem \ref{thm2} by taking an $x$-independent test function.

%$$\star$$

\vspace{10pt}\refstepcounter{sectionno}\S \arabic{sectionno}.
\label{THM3pfsec1}
We now turn to the proof of Theorem \ref{thm3} and study a more general setting. Consider a finite sequence $\omega_1,\ldots,\omega_J$ of positive real numbers (the ``frequencies'').
With each $\omega_j$ we associate a real number $\beta_j$. We are interested in the statistical properties of the 
multiset
\begin{equation}
\scrS = \biguplus_{j=1}^J (\beta_j + \omega_j^{-1} \ZZ) .
\end{equation}
Here $\uplus$ denotes multiset sum, i.e.\ a union where we keep track of the multiplicity of each element (this is only relevant if the sets $\beta_j+\omega_j^{-1}\ZZ$, $j=1,\ldots,J$ are not mutually disjoint).
The number of elements of $\scrS$ that fall into the interval $I=[t-\frac{L}{2},t+\frac{L}{2}]$ is 
\begin{equation}
N(t,L) %= \big| I \cap \scrS \big| 
= \sum_{j=1}^J N_j(t,L) ,\qquad N_j(t,L)=\big| I \cap (\beta_j + \omega_j^{-1} \ZZ) \big| .
\end{equation}
Evidently $N_j(t,L)$ is periodic in $t$ with period $\omega_j^{-1}$.
Furthermore
\begin{equation}\label{est}
\omega_jL-1 <  N_j(t,L)\leq\omega_jL+1.
\end{equation}
Therefore
\begin{equation}
\lim_{L\to\infty}\frac{N(t,L)}{L} =\sum_{j=1}^J\omega_j,
\end{equation}
which means  that the asymptotic density of $\scrS$ in $\RR$ equals $ \sum_j\omega_j$.

For any $x\in\RR$ we write $|x|_\ZZ\in[0,\frac12]$ for the distance between $x$ and the nearest integer. We also set
\begin{equation}
n(x,L) = \big|[x-\tfrac{L}{2},x+\tfrac{L}{2}]\cap \ZZ \big|.
\end{equation}
The verification of the following lemma is a simple exercise.
\begin{lem}\label{COUNTINGLEM}
Given $x\in\RR$, $L\geq0$ and $k\in\ZZ_{\geq0}$,
we have $n(x,L)=k$ if and only if 
$|x-\frac12k|_\ZZ$ is both $\geq\frac12(k-L)$ and $>\frac12(L-k)$.
\end{lem}
In particular, if $x$ is picked at random with respect to the uniform probability measure $\PP_0$ on $[0,1]$,
then
\begin{equation}\label{PP0}
\PP_0 (n(x,L)=k) = 
E_1(k,L):=\begin{cases}
1-|k-L| & \text{if $L-1< k < L+1$}\\
0  & \text{otherwise.}
\end{cases}
\end{equation}

The following observation is due to Barra and Gaspard \cite{Barra00}.

\begin{thm}\label{thm4}
Assume that the frequencies $\omega_1,\ldots,\omega_J$ are linearly independent over $\QQ$. Then, for $-\infty<a<b<\infty$
and $k\in\ZZ_{\geq0}$,
\begin{equation}\label{thm4RES}
\lim_{T\to\infty} \frac{\meas(\{ t\in[aT,bT] : N(t,L)=k \})}{(b-a)T}= E(k,L) 
\end{equation}
where
\begin{equation}
E(k,L)= \sum_{k_1+\ldots+k_J= k} \prod_{j=1}^J E_1\big(k_j,\omega_jL\big) .
\end{equation}
In particular,
\begin{equation}
E(0,L)=
\begin{cases}
\prod_{j=1}^J \big(1-\omega_jL\big) & \text{if $L<\max(\omega_1,\ldots,\omega_J)^{-1}$}\\
0  & \text{otherwise.}
\end{cases}
\end{equation}
The convergence in \eqref{thm4RES} is uniform over all choices of $\beta_1,\ldots,\beta_J\in\RR$
(but keeping $\omega_1,\ldots,\omega_J$ fixed).
\end{thm}
\begin{proof}
Decomposing $[aT,bT]$ into intervals of length $\omega_1^{-1}$
and using $N(t,L)=\sum_{j=1}^JN_j(t,L)$, we see that $\meas(\{ t\in[aT,bT] : N(t,L)=k \})$ equals
\begin{align}
\sum_{k_1+\ldots+k_J=k}\sum_{n=n_a}^{n_b-1}
\omega_1^{-1}\meas(\{ s\in[0,1) : N_j(\omega_1^{-1}(s+n),L)=k_j\text{ for }j=1,\ldots,J \})+E,
%\sum_{k_1+k_2+\ldots+k_J=k}\sum_{n=n_a}^{n_b-1}
%\meas(\{ t\in[n\omega_1^{-1},(n+1)\omega_1^{-1}) : N_j(t,L)=k_j\text{ for }j=1,\ldots,J \})+E,
\end{align}
where $n_a=\lfloor\omega_1 aT\rfloor$, $n_b=\lfloor\omega_1 bT\rfloor$, and $E$ is a real number
whose absolute value is bounded above by the measure of the symmetric difference between 
$[aT,bT]$ and $[n_a\omega_1^{-1},n_b\omega_1^{-1}]$; thus $|E|\leq2\omega_1^{-1}$.
Let us define $\alpha_{j,n}\in\RR/\ZZ$ through the relation
$\beta_j+\omega_j^{-1}\ZZ=n\omega_1^{-1}+\omega_j^{-1}(\alpha_{j,n}+\ZZ)$.
Then $N_j(\omega_1^{-1}(s+n),L)=k_j$ holds if and only if 
$\bigl|[\omega_1^{-1}s-\frac L2,\omega_1^{-1}s+\frac L2]\cap\omega_j^{-1}(\alpha_{j,n}+\ZZ)\bigr|=k_j$,
and by Lemma \ref{COUNTINGLEM} this holds if and only if
$2\bigl|\frac{\omega_j}{\omega_1}s-\alpha_{j,n}-\frac12k_j\bigr|_\ZZ$ is both $\geq k_j-\omega_jL$ and
$>\omega_jL-k_j$.
It follows that 
%Hence, accepting a modification of the condition for a set of $s$-points of measure zero, we conclude that
$\meas(\{ t\in[aT,bT] : N(t,L)=k \})$ equals
\begin{align}
\sum_{k_1+\ldots+k_J=k}\sum_{n=n_a}^{n_b-1}
\omega_1^{-1}f_{k_1,\ldots,k_J}(\alpha_{1,n},\alpha_{2,n},\ldots,\alpha_{J,n})+E,
\end{align}
where $f_{k_1,\ldots,k_J}:(\RR/\ZZ)^J\to[0,1]$ is defined by
\begin{align}
f_{k_1,\ldots,k_J}(\alpha_{1},\alpha_{2},\ldots,\alpha_{J})
=\int_0^1\prod_{j=1}^J I\Bigl(2\Bigl|\frac{\omega_j}{\omega_1}s-\alpha_{j}-\frac12k_j\Bigr|_\ZZ>|k_j-\omega_jL|\Bigr)\,ds.
\end{align}
Note that each function $f_{k_1,\ldots,k_J}$ is continuous on $(\RR/\ZZ)^J$.
It follows from our definition of $\alpha_{j,n}$ that
$\alpha_{j,n+1}=\alpha_{j,n}-\omega_j\omega_1^{-1}\mod\ZZ$.
In particular $\alpha_{1,n}$ is independent of $n$; in fact $\alpha_{1,n}=\alpha_1:=\omega_1\beta_1$ in $\RR/\ZZ$ for all $n$.
On the other hand our assumption implies that the numbers $1,\omega_2\omega_1^{-1},\omega_3\omega_1^{-1},\ldots,
\omega_J\omega_1^{-1}$ are linearly independent over $\QQ$, and hence by Weyl equidistribution we have
\begin{align}
\frac{\sum_{n=n_a}^{n_b-1}f_{k_1,\ldots,k_J}(\alpha_{1,n},\alpha_{2,n},\ldots,\alpha_{J,n})}{n_b-n_a}
\to \int_{(\RR/\ZZ)^{J-1}}f_{k_1,\ldots,k_J}(\alpha_1,\alpha_{2},\ldots,\alpha_{J})\,d\alpha_2\ldots d\alpha_J
\end{align}
as $n_b-n_a\to\infty$.
Note that the convergence here is uniform over all choices of $\beta_1,\ldots,\beta_J\in\RR$.
The right hand side equals
\begin{align}
\int_{(\RR/\ZZ)^{J-1}}\int_0^1\prod_{j=1}^J 
I\Bigl(2\Bigl|\frac{\omega_j}{\omega_1}s-\alpha_{j}-\frac12k_j\Bigr|_\ZZ>|k_j-\omega_jL|\Bigr)\,ds,
\end{align}
and changing order of integration and then substituting $\alpha_j:=x_j+\frac{\omega_j}{\omega_1}s-\frac12k_j$
($x_j\in\RR/\ZZ$) we see that the expression factors as $\prod_{j=1}^J E_1(k_j,\omega_jL)$.
Finally noticing also that $n_b-n_a\sim\omega_1(b-a)T$ as $T\to\infty$, we obtain \eqref{thm4RES}.
\end{proof}

We now turn to the gap distribution. We order the elements of $\scrS$ and label them as 
\begin{equation}
\ldots \leq \lambda_{-2} \leq \lambda_{-1} \leq \lambda_0 \leq \lambda_1\leq \lambda_2 \leq \ldots  .
\end{equation}

\begin{thm}\label{mainthm}
Assume that the frequencies $\omega_j$ are linearly independent over $\QQ$,
and that $\omega_1$ is the largest among $\omega_1,\ldots,\omega_J$. 
Then, for $-\infty<a<b<\infty$, and any $s\geq0$, $s\neq\omega_1^{-1}$,
\begin{equation}\label{mainthmres}
\lim_{T\to\infty} \frac{1}{(b-a)T}| \{ j \in\ZZ %\cap[aT,bT]
:\lambda_j\in[aT,bT]\text{ and } \lambda_{j+1}-\lambda_j > s \}|
= \int_s^\infty P_\omega(s') ds'
\end{equation}
where
\begin{equation}
P_\omega(s) = 
\omega_1 \prod_{j=2}^J \Big(1-\frac{\omega_j}{\omega_1}\Big) \delta(s-\omega_1^{-1})
+\begin{cases}
\displaystyle
\sum_{\substack{h,i=1\\ h\neq i}}^J \omega_h\omega_i \bigg[ \prod_{\substack{j=1\\ j\neq h,i}}^J \big(1-\omega_js\big)\bigg] & \text{if $0< s< \omega_1^{-1}$}\\[5pt]
0 & \text{if $s > \omega_1^{-1}$.}
\end{cases}
\end{equation}
The convergence in \eqref{mainthmres} is uniform over all choices of $\beta_1,\ldots,\beta_J\in\RR$
(but keeping $\omega_1,\ldots,\omega_J$ fixed).
\end{thm}

\begin{proof}
This is a corollary of Theorem \ref{thm4}, cf.~e.g.~Theorem 2.2 in \cite{nato}.
The limiting gap density is given by the formula
\begin{equation}
P_\omega(s) = \frac{d^2 E(0,s)}{ds^2} .
\end{equation}
\end{proof}

%$$\star$$

\vspace{10pt}\refstepcounter{sectionno}\S \arabic{sectionno}.
\label{THM3pfsec2}
%\begin{proof}[Proof of Theorem \ref{thm3}]
\textit{Proof of Theorem \ref{thm3}.}
It suffices to prove that \eqref{thm3res} holds when $f$ is the characteristic function of a box
$[0,a]\times[0,A]$ with $A\neq(1-b^{-1})^{-1}$,
i.e.\ to prove that for any fixed $a\in[0,1]$ and $A\geq0$, $A\neq(1-b^{-1})^{-1}$, we have
\begin{align}\label{thm3pf1}
G_{a,A}(N)\to a\int_0^A \tilde P(s)\,ds
\end{align}
as $N\to\infty$, where
\begin{align}
G_{a,A}(N):=
\frac1N\Big|\Big\{n\in\{1,\ldots,N\} :\tilde\eta_{n,N}\in[0,a],\: N(\tilde\eta_{n+1,N}-\tilde\eta_{n,N})\in[0,A]\Big\}\Big|
\end{align}
(cf., e.g., \cite[Example 2.3]{Billingsley}).

Set $\omega_j=b^{-j}(b-1)$ and $\beta_j^{(N)}=\beta^{(N)}=-\frac N{b-1}$ for $j=1,2,\ldots$.
For any given $J,N\in\ZZ_{\geq1}$ we note that by \eqref{keyeq}, 
%the multiset 
$\{N\tilde\eta_1,\ldots,N\tilde\eta_N\}$ equals the multiset sum of
$(\beta^{(N)}+\omega_j^{-1}\ZZ)\cap[0,N)$ for $j=1,2,\ldots,J$ and a remaining multiset 
of cardinality $\lceil Nb^{-J}\rceil$ consisting of $N\tilde\eta_n$ for $n=N$ and all $n<Nb^{-J}$.
Hence if we order the elements of $\scrS_N:=\uplus_{j=1}^J(\beta^{(N)}+\omega_j^{-1}\ZZ)$ as
$\ldots\leq\lambda_{-2}\leq \lambda_{-1} \leq \lambda_0 \leq \lambda_1\leq\lambda_2\leq \ldots$, then
\begin{align}
G_{a,A}(N)=\frac1N\Big|\Big\{j\in\ZZ :\lambda_j\in[0,aN],\: \lambda_{j+1}-\lambda_j\in[0,A]\Big\}\Big|
+O\bigl(N^{-1}+b^{-J}\bigr).
\end{align}
%since each ``new'' cannot decrease the number of gaps counted, but can create one or two new gaps counted.
Since $b$ is transcendental, the frequencies $\omega_1,\ldots,\omega_J$ are linearly independent over $\QQ$;
thus Theorem \ref{mainthm} applies, and we obtain, for any fixed $J\in\ZZ_{\geq1}$,
\begin{align}\label{thm3pf2}
&\limsup_{N\to\infty}G_{a,A}(N)\leq a\int_0^A P_{\omega,J}(s')\, ds'+O(b^{-J});
\\\notag
&\liminf_{N\to\infty}G_{a,A}(N)\geq a\int_0^A P_{\omega,J}(s')\, ds'-O(b^{-J}),
\end{align}
where $P_{\omega,J}(s)=0$ if $s>\omega_1^{-1}$ %=\frac b{b-1}$ 
while for $s<\omega_1^{-1}$ we have
\begin{align}
P_{\omega,J}(s)
%=\omega_1 \prod_{j=2}^J \Big(1-\frac{\omega_j}{\omega_1}\Big) \delta(s-\omega_1^{-1})
%+ \sum_{\substack{h,i=1\\ h\neq i}}^J \omega_h\omega_i \bigg[ \prod_{\substack{j=1\\ j\neq h,i}}^J \big(1-\omega_js\big)\bigg]
%\\
=\omega_1 \prod_{j=2}^J \Big(1-\frac{\omega_j}{\omega_1}\Big) \delta(s-\omega_1^{-1})
+\frac{d^2}{ds^2}\prod_{j=1}^J \big(1-\omega_js\big).
%+ \prod_{j=1}^J \bigg(1-\omega_js\bigg) \sum_{\substack{h,i=1\\ h\neq i}}^J \frac{1}{(\omega_h^{-1}-s)(\omega_i^{-1}-s)} ,
\end{align}
We now compute that
\begin{align}
\int_0^A P_{\omega,J}(s')\, ds'=\sum_{j=1}^J\omega_j-
\begin{cases}0  &\text{if }A>\omega_1^{-1}
\\[5pt]
{\displaystyle\sum_{j=1}^J\omega_j\prod_{\substack{i=1\\(i\neq j)}}^J(1-\omega_iA)} 
&\text{if }A<\omega_1^{-1}.
\end{cases}
\end{align}
Hence, letting $J\to\infty$ in \eqref{thm3pf2} and using $\omega_j=b^{-j}(b-1)$, we conclude
\begin{align}
\lim_{N\to\infty} G_{a,A}(N)=a+a\left.
\begin{cases}0  &\text{if }A>\frac b{b-1}
\\[5pt]
{\displaystyle\Bigl(\frac{\partial}{\partial s}\bigl((1-b^{-1})s;b^{-1}\bigr)\Bigr)_{|s=A}}&\text{if }A<\frac b{b-1}
\end{cases}
\right\}
=a\int_0^A \tilde P(s)\,ds,
\end{align}
where the last equality follows by a direct computation using \eqref{Rab} and \eqref{thm3tildePdef}.
Hence \eqref{thm3pf1} holds, and Theorem \ref{thm3} is proved.
%\end{proof}
\hfill$\square$

%$$\star$$

\vspace{10pt}\refstepcounter{sectionno}\S \arabic{sectionno}.
\label{THM2EQTHM3para}
Let us now prove that Theorem \ref{thm3} implies Theorem \ref{thm2}.
(An entirely analogous argument also shows that Theorem \ref{thm2} implies Theorem \ref{thm3},
so that the statements of these two theorems are in fact equivalent.)
Given a bounded continuous function $f:[0,1]\times\RR_{\geq0}\to\RR$,
we wish to prove that \eqref{thm2res} holds.
Since the sequence of probability measures $P_N(x,s)\,ds\,dx$ is tight
(cf.,\ e.g.,\ Lemma 2.1 in \cite{nato}),
%$\int_0^1\int_0^\infty P(x,s)\,ds\,dx=1$
%(which is immediate from \eqref{Pxs}, \eqref{thm3tildePdef}, \eqref{Rab}),
a familiar %standard 
approximation argument shows that without loss of generality we may assume that $f$ has compact support,
i.e.\ there is some $B>0$ such that $f(x,s)=0$ whenever $s\geq B$.

Let $T:[0,1]\to[0,1]$ be the map in \eqref{TMAP} and define $\fT:[0,1]\times\RR_{\geq0}\to[0,1]\times\RR_{\geq0}$ through
$\fT(x,s)=(T(x),T'(x)s)$. %=(T(x),\rho(x)s).
This is a bijection, with inverse given by
\begin{align}
\fT^{-1}(x,s)=\Bigl(T^{-1}(x),\frac s{T'(T^{-1}(x))}\Bigr).
\end{align}
Now apply Theorem \ref{thm3} with the test function $f\circ\fT^{-1}$.
Then in %the left hand side of 
\eqref{thm3res} we have
\begin{align}
\int_0^1 \int_0^\infty f\circ\fT^{-1}(x,s)\, \tilde P_N(x,s)\,ds\,dx
%=\frac1N\sum_{n=1}^N f\circ\fT^{-1}\bigl(\tilde\eta_{n,N},N(\tilde\eta_{{n+1},N}-\tilde\eta_{n,N})\bigr)
%\\
=\frac1N\sum_{n=1}^N f(\eta_{n,N},s_{n,N}),
\end{align}
where
${\displaystyle s_{n,N}:=N\cdot\frac{T(\eta_{{n+1},N})-T(\eta_{n,N})}{T'(\eta_{n,N})}}.$
Let $\scrF_N$ be the set of those $n\in\{1,\ldots,N\}$ for which $\eta_{{n+1},N}-\eta_{n,N}\leq BbN^{-1}$.
Then since $T'(x)=\rho(x)$ is continuous and bounded away from zero on $[0,1]$ we have
$s_{n,N}=N(\eta_{n+1,N}-\eta_{n,N})+o(1)$
uniformly over all $n\in\scrF_N$ as $N\to\infty$,
and hence, since $f$ is uniformly continuous, 
\begin{align}\label{THM3IMPLTHM2pf1}
f(\eta_{n,N},s_{n,N})=f(\eta_{n,N},N(\eta_{n+1,N}-\eta_{n,N}))+o(1),
\end{align}
uniformly over all $n\in\scrF_N$.
On the other hand if $n\notin\scrF_N$ then $N(\eta_{n+1,N}-\eta_{n,N})>Bb>B$ and also
$s_{n,N}>Bb\frac{\inf T'}{\sup T'}=B$,
so that both sides of \eqref{THM3IMPLTHM2pf1} vanish.
Hence \eqref{THM3IMPLTHM2pf1} in fact holds uniformly over \textit{all} $n\in\{1,\ldots,N\}$ as $N\to\infty$,
and it follows that the limit in the left hand side of \eqref{thm2res} exists, and equals the limit in \eqref{thm3res}
(with $f\circ\fT^{-1}$).
Finally substituting $(x,s)=\fT(x',s')$ in the right hand side of \eqref{thm3res} we see
(via \eqref{Pxs}) that this limit equals the right hand side of \eqref{thm2res}, and the proof is complete.

%$$\star$$

\vspace{10pt}\stepcounter{sectionno}\S \arabic{sectionno}.
Theorem \ref{thm1} is now a direct consequence of Theorem \ref{thm2}. As to the explicit formula for the gap distribution, 
\begin{equation}
\begin{split}
P(s) & = \int_0^1 P(x,s)\, dx \\
& =\int_0^1 (b^{x-1}\log b)^2 R(s b^{x-1}\log b, b) \, dx .
\end{split}
\end{equation}
With the variable substitution $a=s b^{x-1}\log b$, $da=s b^{x-1} (\log b)^2\, dx=(\log b) \, a\, dx$ we have
\begin{equation}
P(s)  = \frac{1}{s^2\log b} \int_{s b^{-1}\log b}^{s\log b}  R(a, b) \, a\, da .
\end{equation}
Recall the definition of $R(a,b)$ in \eqref{Rab}.
The integral over the first term in \eqref{Rab} yields $\frac{(b^{-1};b^{-1})}{s^2\log b}$ if $s b^{-1}\log b < 1< s\log b$ and $0$ otherwise. For the second term, we use integration by parts,
\begin{equation}
\begin{split}
\int_{a_0}^{a_1} a \frac{\partial^2}{\partial a^2} (a;b^{-1})\, da & = \bigg[a \frac{\partial}{\partial a} (a;b^{-1})\bigg]_{a_0}^{a_1} - \int_{a_0}^{a_1} \frac{\partial}{\partial a} (a;b^{-1})\, da \\
& = \bigg[a \frac{\partial}{\partial a} (a;b^{-1})- (a;b^{-1})\bigg]_{a_0}^{a_1} ,
\end{split}
\end{equation}
where $a_0=\min\{ 1, s b^{-1}\log b\}$ and $a_1=\min\{ 1, s \log b\}$. This yields \eqref{Ps}.

%$$\star$$

\begin{figure}
\begin{center}
\includegraphics[width=0.7\textwidth]{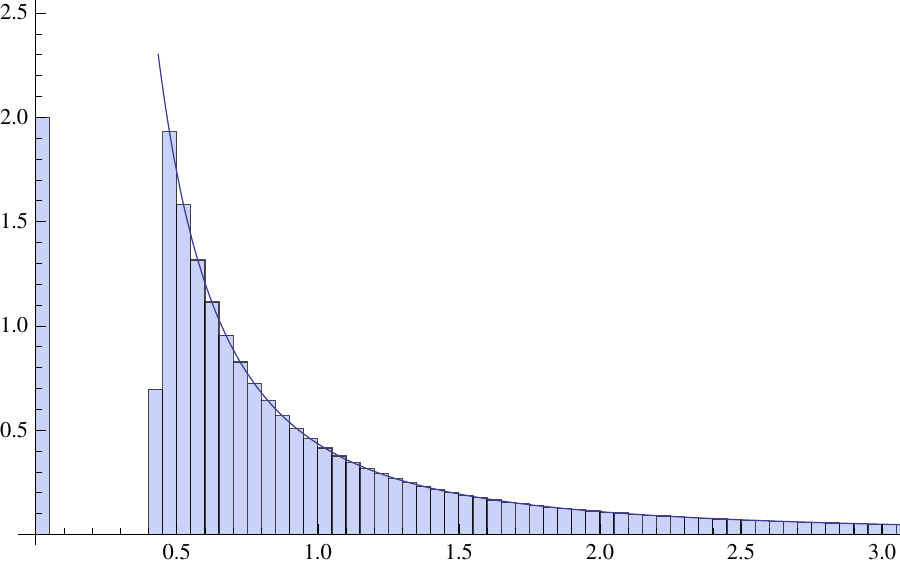}
\end{center}
\caption{The distribution of gaps between the fractional parts of $\log_b n$ with base $b=10$, where $n=1,\ldots,10^4$. The piecewise continuous curve is the limit distribution $P(s)$ in \eqref{Ps10}.} \label{fig3}
\end{figure}

\begin{figure}
\begin{center}
\includegraphics[width=0.7\textwidth]{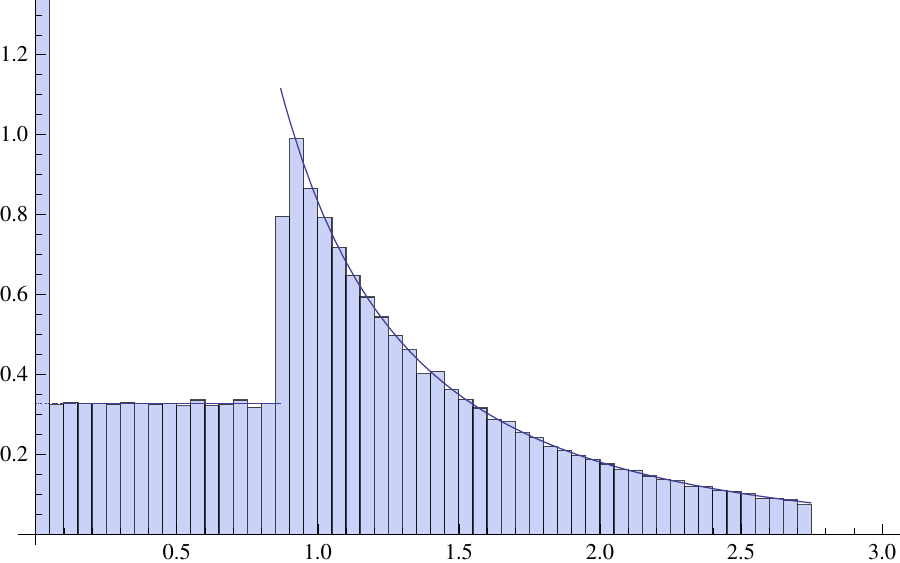}
\end{center}
\caption{The distribution of gaps between the fractional parts of $\log_b n$ with base $b=\sqrt{10}$, where $n=1,\ldots,10^4$. The piecewise continuous curve is the limit distribution $P(s)$ in \eqref{Ps_r} with $r=2$.} \label{fig4}
\end{figure}

\begin{figure}
\begin{center}
\includegraphics[width=0.7\textwidth]{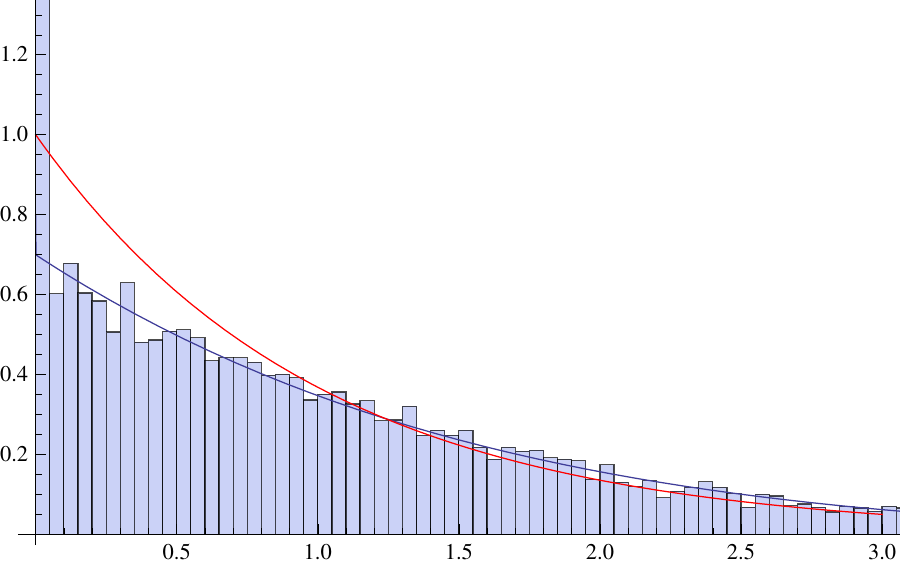}
\end{center}
\caption{The distribution of gaps between the fractional parts of $10\log_{10} n$ ($=\log_b n$ with base $b=10^{1/10}=1.258925411\ldots$), where $n=1,\ldots,10^4$. The blue curve is the limit distribution $P(s)$ in \eqref{Ps_r} with $r=10$. The red curve is the exponential distribution $\e^{-s}$.} \label{fig5}
\end{figure}

\vspace{10pt}\refstepcounter{sectionno}\S \arabic{sectionno}.
\label{BINTEGER}
As we have noted, the technique presented here works also for algebraic $b$.
Let us consider briefly the simplest case of $b>1$ being an integer.
In this case, the multiset sum $\scrS_N=\uplus_{j=1}^J(\beta^{(N)}+\omega_j^{-1}\ZZ)$ 
considered in the proof of Theorem \ref{thm3} has all its points lying in $\beta^{(N)}+\omega_1^{-1}\ZZ$,
i.e.\ the only gap lengths appearing are $0$ and $\omega_1^{-1}=(1-b^{-1})^{-1}$;
furthermore $\scrS_N$ is periodic with period $\omega_J^{-1}=b^{J-1}\omega_1^{-1}$.
One finds that the average multiplicity of $\beta^{(N)}+\omega_1^{-1}k$ in $\scrS_N$ for varying $k\in\ZZ$
equals $\frac{1-b^{-J}}{1-b^{-1}}$.
Hence following the proof of Theorem \ref{thm3} and taking the limit $J\to\infty$, we conclude that for $b>1$ an
integer, we have a limit result as in \eqref{thm3res} in Theorem \ref{thm3}, but with the limit density being given simply by
\begin{align}
\tilde P(s)=b^{-1}\delta(s)+(1-b^{-1})\delta\bigl(s-(1-b^{-1})^{-1}\bigr).
\end{align}
As to the limit of the unscaled sequence $\eta_n$, the first equality in \eqref{Pxs} yields
\begin{equation}
P(x,s)=\frac{\log b}{b-1}\, b^{x-1} \delta(s) + (\log b) b^{x-1} \delta\biggl(s-\frac{b^{1-x}}{\log b}\biggr) ,
\end{equation}
and therefore
\begin{equation}\label{Ps10}
P(s)=b^{-1}  \delta(s) + 
\begin{cases} 
\frac{1}{\log b}\, s^{-2}  & \text{ if $\frac{1}{\log b}<s<\frac{b}{\log b}$} \\
0 & \text{ otherwise.}
\end{cases}
\end{equation}

\vspace{10pt}\refstepcounter{sectionno}\S \arabic{sectionno}.
\label{NEWSEC}
The findings of the previous paragraph can be generalised to bases of the form $b=m^{1/r}$, with $m$ and $r$ positive integers such that $m^{1/p}$ is not an integer for any prime divisor $p$ of $r$. In this case, the multiset sum $\scrS_N=\uplus_{j=1}^J(\beta^{(N)}+\omega_j^{-1}\ZZ)$ 
considered in the proof of Theorem \ref{thm3} has all points in the set $\cup_{j=1}^r (\beta^{(N)}+\omega_j^{-1}\ZZ)$. Our assumption implies that the polynomial $X^r-m$ is irreducible over $\QQ$ (since any monic polynomial of degree $d<r$ dividing $X^r-m$ must have zeroth coefficient of absolute value $m^{d/r}$, which is irrational). Hence $\omega_1,\ldots,\omega_r$ are linearly independent over $\QQ$, and we now see by the same argument as in the proof of Theorem \ref{thm4} that the limit \eqref{thm4RES} exists in the case $k=0$ and is given by
\begin{equation}
E(0,L)=
\begin{cases}
\prod_{j=1}^r \big(1-b^{-j}(b-1) L\big) & \text{if $L<(1-b^{-1})^{-1}$}\\
0  & \text{otherwise.}
\end{cases}
\end{equation}
Using the finite q-Pochhammer symbol 
\begin{equation}
(a;q)_r = \prod_{n=0}^{r-1} (1-aq^n)  ,
\end{equation}
we can write this as
\begin{equation}
E(0,L)=
\begin{cases}
((1-b^{-1}) L; b^{-1})_r & \text{if $L<(1-b^{-1})^{-1}$}\\
0  & \text{otherwise.}
\end{cases}
\end{equation}
We note that for $L<(1-b^{-1})^{-1}$
\begin{equation}
-\frac{d}{dL} E(0,L) = (1-b^{-1}) E(0,L) \sum_{j=0}^{r-1} \frac{1}{b^j-(1-b^{-1}) L} 
\end{equation}
and thus
\begin{equation}
-\frac{d}{dL} E(0,L) \bigg|_{L=0}= (1-b^{-1}) \sum_{j=0}^{r-1} \frac{1}{b^j} = 1-b^{-r}.
\end{equation}
The fact that this value is less than $1$ is due to the non-trivial multiplicity of values in our sequence, which results in a positive density of zero gaps. The statements of Theorems \ref{thm1}--\ref{thm3} therefore hold with the following limit distributions. The limiting gap distribution of Theorem \ref{thm1} is
\begin{equation}\label{Ps_r}
P(s)=b^{-r} \, \delta(s) + \begin{cases}
 \frac{1}{\log b}\, s^{-2} (F_r(s\log b)-F_r(sb^{-1}\log b)) & \text{if }\:0<s<\frac{1}{\log b} \\
-\frac{1}{\log b}\, s^{-2}F_r(sb^{-1}\log b) & \text{if }\: \frac{1}{\log b} < s<\frac{b}{\log b} \\
0 & \text{if }\: s>\frac{b}{\log b},
\end{cases}
\end{equation}
with
\begin{equation}
F_r(a) = a \frac{\partial}{\partial a} (a;b^{-1})_r- (a;b^{-1})_r .
\end{equation}
The joint limiting gap distribution of Theorem \ref{thm2} is
\begin{equation}
P(x,s)=\frac{\log b}{b-1}\,b^{x-r} \, \delta(s) + (b^{x-1}\log b)^2 R_r(s b^{x-1}\log b, b) 
\end{equation}
and the rescaled gap distribution of Theorem \ref{thm3} is
\begin{equation}\label{thm3tildePdef_r}
\tilde P(s) = b^{-r} \, \delta(s) + (1-b^{-1})^2 R_r\big((1-b^{-1}) s, b\big) ,
\end{equation}
with
\begin{equation}\label{Rab_r}
R_r(a,b) = 
(b^{-1};b^{-1})_{r-1}\,   \delta(a-1)+
\begin{cases}
\displaystyle
\frac{\partial^2}{\partial a^2} (a;b^{-1})_r  & \text{if $0<a<1$}\\[5pt]
 0 & \text{if $a>1$.}
\end{cases}
\end{equation}

%$$\star$$

\vspace{10pt}\refstepcounter{sectionno}\S \arabic{sectionno}.
\label{LASTPARA}
We conclude this investigation with a few remarks on the family of limit distributions which we have obtained for different values of $b\in\scrT$. It is helpful to introduce the random point process in $\RR$ given by the sequence
\begin{equation}
\varphi_N^{(b)} = \{ N(\tilde\eta_n + t + m) : 1\leq n\leq N,\; m\in\ZZ\}
\end{equation}
where $t$ is uniformly distributed in $[0,1]$. This process is clearly stationary and has intensity one. By using the same strategy as in the proof of Theorem \ref{thm3}, one can deduce from Theorem \ref{thm4} a limit law for the one-dimensional distribution of $\varphi_N^{(b)}$: For every closed interval $A\subset\RR$, $k\in\ZZ_{\geq 0}$,
\begin{equation}\label{con1}
\lim_{N\to\infty} \PP( |\varphi_N^{(b)} \cap A | = k ) = E^{(b)}(k,\meas(A)),
\end{equation}
where 
\begin{equation}\label{con2}
E^{(b)}(k,L)= \lim_{J\to\infty} \sum_{k_1+\ldots+k_J= k} \prod_{j=1}^J E_1\big(k_j,(b-1)b^{-j} L\big) ,
\end{equation}
with $E_1(k,L)$ as in \eqref{PP0}. In the case $k=0$, relation \eqref{con1} in fact follows directly from Theorem \ref{thm3} (exploiting again Theorem 2.2 in \cite{nato}), where
\begin{equation}
\tilde P^{(b)}(s) = \frac{d^2 E^{(b)}(0,s)}{ds^2} .
\end{equation}
Although \eqref{con1} requires $b\in\scrT$, the family of distributions \eqref{con2} is well defined for any real $b\in(1,\infty)$. Let us briefly analyse the limiting cases $b\to\infty$ and $b\to 1$.

For $b\to\infty$,  $E^{(b)}(k,L)$ converges to the statistics of the point process given by the randomly shifted integer lattice $\ZZ+t$, with $t$ uniformly distributed in $[0,1]$. That is, $E^{(b)}(k,L) \to E_1(k,L)$ as in \eqref{PP0}. The rescaled gap distribution satisfies
\begin{equation}
\tilde P^{(b)}(s) \to R(s,\infty)=\delta(s-1) .
\end{equation}
The limit of the raw $P^{(b)}(s)$ distribution is slightly more involved since the asymptotic density $\rho^{(b)}(x)$ in \eqref{rho} converges to $\delta(x-1)$. We therefore cannot expect $P^{(b)}(s)$ to converge to a non-trivial limit without further rescaling. Let us extend $\rho^{(b)}(x)$ to a probability density on $\RR_{\geq 0}$ by setting $\rho^{(b)}(x)=0$ for $x>1$. We have the scaling limit
\begin{equation}
\frac{b-1}{\log b} \rho^{(b)}\bigg(\frac{x}{\log b} \bigg) \to \e^x ,
\end{equation}
which, however, is not a probability density on $\RR_{\geq 0}$. Working directly from \eqref{Ps}, we see that for $s$ fixed,
\begin{equation}
\frac{1}{\log b}\; P^{(b)}\bigg(\frac{s}{\log b}\bigg) \to
\begin{cases}
0 & \text{if }0<s <1 \\
s^{-2} &\text{if }s>1.
\end{cases}
\end{equation}

In the limit $b\to 1$, $E^{(b)}(k,L)$ converges to the Poisson distribution. That is,
\begin{equation}
E^{(b)}(k,L) \to \frac{L^k}{k!} \e^{-L}, \qquad \tilde P^{(b)}(s) \to \e^{-s} .
\end{equation}
To see this, note that for $L$ fixed and $b$ sufficiently close to $1$,
\begin{equation}
\log E^{(b)}(0,L) = \sum_{j=1}^\infty \log(1-(b-1)b^{-j} L) = -L +O(b-1) ,
\end{equation}
which proves the claim for $k=0$ and the gap distribution $\tilde P(s)$ (via Theorem 2.2 in \cite{nato}).
Again, for $L$ fixed and $b$ sufficiently close to $1$, we have for the remaining cases $k\geq 1$,
\begin{equation}
\begin{split}
E^{(b)}(k,L) & = E^{(b)}(0,L) \sum_{j_1<\ldots<j_k} \frac{E_1(1,(b-1)b^{-j_1} L)\cdots E_1(1,(b-1)b^{-j_k} L)}{E_1(0,(b-1)b^{-j_1} L)\cdots E_1(0,(b-1)b^{-j_k} L)} \\
&= (b-1)^k L^k E^{(b)}(0,L) \sum_{j_1<\ldots<j_k} \frac{b^{-(j_1+\ldots+j_k)}}{(1-(b-1)b^{-j_1} L)\cdots (1-(b-1)b^{-j_k} L)} \\
& = L^k E^{(b)}(0,L) \;\bigg(\frac{1}{k!}+O(b-1)\bigg).
\end{split}
\end{equation}

As to the limit of the unscaled sequence $\eta_n$,
%raw $P^{(b)}(s)$ distribution,
%fractional parts of $\log_b n$, 
we have for the density
\begin{equation}
\rho^{(b)}(x) \to 1, \qquad x\in[0,1] ,
\end{equation}
i.e., we have uniform distribution mod one. In view of \eqref{Pxs}, $P^{(b)}(x,s) \to \e^{-s}$ and thus, for the raw gap distribution, $P^{(b)}(s) \to \e^{-s}$.

The above discussion can be readily adapted to bases of the form $b=m^{1/r}$ considered in \S\ref{NEWSEC}. In particular, we again observe that both the raw and rescaled limiting gap distributions \eqref{Ps_r} and \eqref{thm3tildePdef_r} converge to the exponential distribution $\e^{-s}$ when $r\to\infty$ for fixed $m$, cf. Fig.~\ref{fig5}.

\vspace{10pt} {\em Acknowledgments.} We thank Jon Keating and Ze\'ev Rudnick for their comments on the first draft of this paper.

\end{document}